\documentclass[12pt]{article}
\usepackage{fullpage,amssymb,amsfonts,amsmath,amstext,amsthm,amscd,graphicx,graphics,hyperref}

\begin{document}

\newtheorem{theorem}{Theorem}[section]
\newtheorem{result}[theorem]{Result}
\newtheorem{fact}[theorem]{Fact}
\newtheorem{proposition}[theorem]{Proposition}
\newtheorem{definition}[theorem]{Definition}
\newtheorem{lemma}[theorem]{Lemma}
\newtheorem{remark}[theorem]{Remark}
\newtheorem{corollary}[theorem]{Corollary}
\newtheorem{facts}[theorem]{Facts}
\newtheorem{props}[theorem]{Properties}

\newtheorem{ex}[theorem]{Example}

\newcommand{\notes} {\noindent \textbf{Notes.  }}
\newcommand{\note} {\noindent \textbf{Note.  }}
\newcommand{\defn} {\noindent \textbf{Definition.  }}
\newcommand{\defns} {\noindent \textbf{Definitions.  }}
\newcommand{\x}{{\bf x}}
\newcommand{\z}{{\bf z}}
\newcommand{\B}{{\bf b}}
\newcommand{\V}{{\bf v}}
\newcommand{\T}{\mathbb{T}}
\newcommand{\Z}{\mathbb{Z}}
\newcommand{\Hp}{\mathbb{H}}
\newcommand{\D}{\mathbb{D}}
\newcommand{\R}{\mathbb{R}}
\newcommand{\C}{\mathbb{C}}
\newcommand{\N}{\mathbb{N}}
\newcommand{\dt}{{\mathrm{det }\;}}
 \newcommand{\adj}{{\mathrm{adj}\;}}
 \newcommand{\0}{{\bf O}}
 \newcommand{\av}{\arrowvert}
 \newcommand{\zbar}{\overline{z}}

\newcommand{\ds}{\displaystyle}
\numberwithin{equation}{section}

\title{Quasiregular mappings of polynomial type in $\R^{2}$}

\author{Alastair Fletcher and Dan Goodman \footnote{The first author is supported by EPSRC grant EP/G050120/1.} }

\maketitle

\begin{abstract}

Complex dynamics deals with the iteration of holomorphic functions. As is well-known, the first functions to be studied which gave non-trivial dynamics were quadratic polynomials,
which produced beautiful computer generated pictures of Julia sets and the Mandelbrot set. In the same spirit, this article aims to study the dynamics of the simplest non-trivial quasiregular mappings. These are mappings in $\R^{2}$ which are a composition of a quadratic polynomial and an affine stretch. 


\end{abstract}

\begin{figure}
  \caption{The set of non-escaping points for $f(z) = (h_{0.8,0}(z))^{2}-0.21-0.78i$, a Douady dragon.}\label{0.8K_-0.21-0.78icfig}
\begin{center}
    \includegraphics{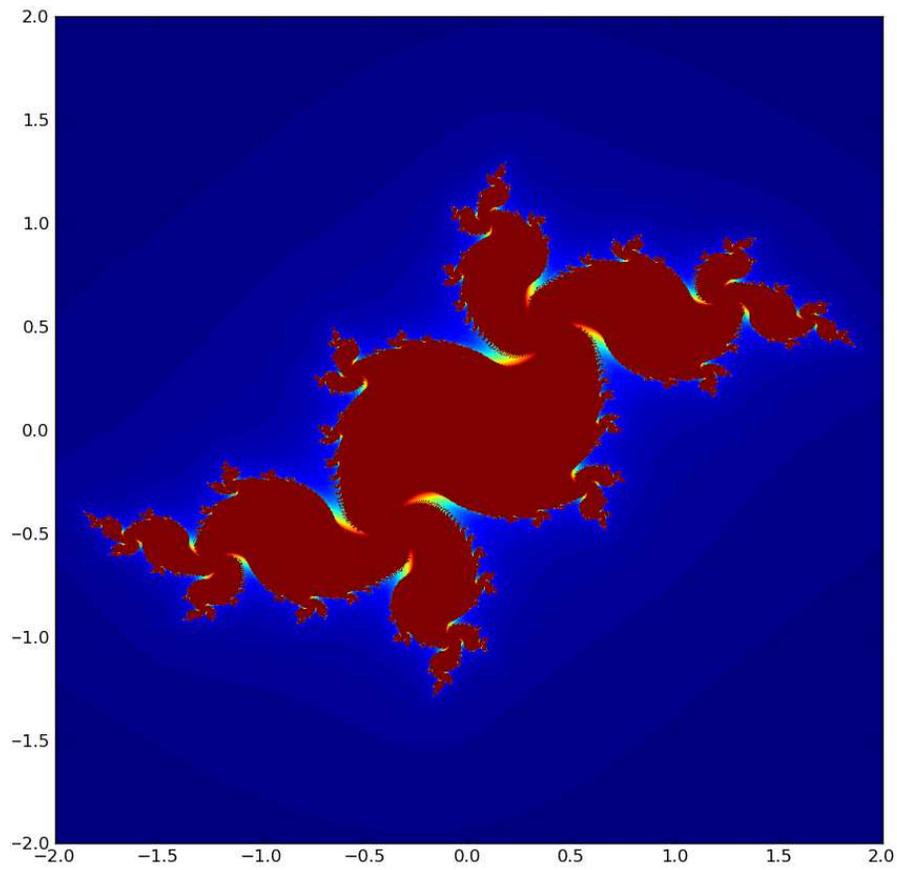}
\end{center}
\end{figure}

\section{Introduction}

The field of complex dynamics, the study of iteration of analytic functions in the plane, goes back nearly a century to Fatou and Julia. However, there has been a surge of recent interest in the field, following on from computer generated pictures of Julia sets and the Mandelbrot set and led by the work of Douady and Hubbard, e.g. \cite{DH}. This illustrated how an iterative system with a very simple description, namely a quadratic polynomial, could have very complicated behaviour. There are several excellent introductions to the theory, for example \cite{Beardon, Milnor, MNTU}.

More recently, the iteration of quasiregular mappings in $\R^{n}$ has been studied, motivated by the fact that several key tools in complex dynamics have analogues for quasiregular mappings, for example Rickman's theorem generalizing Picard's theorem, and Montel's theorem. In fact, direct analogues of the Fatou and Julia sets can be defined for a special class of quasiregular mappings, all of whose iterates have distortion bounded by some fixed number. These are the so-called uniformly quasiregular mappings, introduced in \cite{IM} and studied in a number of papers. We restrict ourselves to mentioning \cite{HMM} and the interested reader can find further references contained therein.

For general quasiregular mappings, it is no longer possible to define the Fatou set for the simple reason that the iterates will have no common bound on the distortion (see the standard references \cite{Rickman,IM2} for the theory of quasiregular mappings). It is however possible to define the escaping set $I(f)$, the set of those points which iterate to infinity, which is a key object in complex dynamics. It is well known that for an analytic function, the boundary of $I(f)$ coincides with the Julia set. It is therefore natural to consider $\partial I(f)$ for quasiregular mappings and see to what extent it can be considered an analogue of the Julia set. Recent papers in this direction include \cite{Bergweiler,BE,BFLM}.

Of particular relevance here is \cite{FN}, where it was shown that for quasiregular mappings of polynomial type, as long as the degree of the mapping is larger than the distortion, $\partial I(f)$ is an infinite, completely invariant perfect set. Further, the family of iterates of $f$ is not equicontinuous at any point of $\partial I(f)$.

In this paper, in the same spirit as the study of iteration of quadratic polynomials, we aim to analyze the boundary of the escaping set for the simplest quasiregular mappings with non-trivial dynamics; namely the composition of quadratic polynomials and quasiconformal mappings with constant complex dilatation. 

The outline of the paper is as follows. In section 2, we will cover some preliminary definitions and results. In section 3, a canonical form for the type of functions we are interested in will be derived. 
In section 4, generalizations of results in \cite{FN}, on the boundary of the escaping set, to our functions will be obtained. 
Section 5 deals with the connectedness or not of $\partial I(f)$, and section 6 introduces a generalization of the Mandelbrot set and covers some of its properties. 

The images in this paper were computed using the Python programming language and the 
NumPy (Numerical Python) extension package. The code is available on the 
second authors' website
\url{http://thesamovar.net/mathematics/qrdynamics}.

The authors would like to thank Dan Nicks for interesting and stimulating discussions.

\section{Preliminaries}

We first collect some definitions and results that we will use. A quasiregular mapping $f:G \rightarrow \R^{n}$ from a domain $G \subseteq \R^{n}$ is called quasiregular if $f$ belongs to the Sobolev space $W^{1}_{n, loc}(G)$ and there exists $K \in [1, \infty)$ such that 
\begin{equation}
\label{eq2.1}
\av f'(x) \av ^{n} \leq K J_{f}(x)
\end{equation}
almost everywhere in $G$. Here $J_{f}(x)$ denotes the Jacobian determinant of $f$ at $x \in G$. The smallest constant $K \geq 1$ for which (\ref{eq2.1}) holds is called the outer distortion $K_{O}(f)$ of $f$. If $f$ is quasiregular, then we also have
\begin{equation}
\label{eq2.2}
J_{f}(x) \leq K' \inf _{\av h \av =1} \av f'(x) h \av ^{n}
\end{equation}
almost everywhere in $G$ for some $K' \in[1, \infty)$. The smallest constant $K' \geq 1$ for which (\ref{eq2.2}) holds is called the inner distortion $K_{I}(f)$ of $f$. The maximal distortion $K=K(f)$ of $f$ is the larger of $K_{O}(f)$ and $K_{I}(f)$. In dimension $2$, we have $K_{O}(f) = K_{I}(f)$.

The degree of a mapping is the maximal number of pre-images and is in direct analogue with the degree of a polynomial. A quasiregular mapping is said to be of polynomial type if its degree is uniformly bounded at every point, or equivalently, if $\av f(x) \av \rightarrow \infty$ as $\av x \av \rightarrow \infty$. 

\begin{theorem}[\cite{FN}]
\label{FN}
Let $n \geq 2$ and $f:\R^{n} \rightarrow \R^{n}$ be $K$-quasiregular of polynomial type. If the degree of $f$ is greater than the inner distortion $K_{I}$, then $I(f)$ is a non-empty open set and $\partial I(f)$ is perfect.
\end{theorem}

Denote by $B(f)$ the branch set of $f$, that is, the set where $f$ is not locally injective.
A quasiconformal mapping is an injective quasiregular mapping. The following result says that in dimension $2$, a quasiregular mapping can be factorized into two mappings, one of which deals with the distortion and one which deals with the branch points.

\begin{theorem}[Stoilow factorization, see for example \cite{IM2} p.254]
\label{Stoilow}
Let $f:\C \rightarrow \C$ be a quasiregular mapping. Then there exists an analytic function $g$ and a quasiconformal mapping $h$ such that $f = g \circ h$.
\end{theorem}

Stoilow factorization tells us what the branch set of a quasiregular mapping in $\C$ can be.

\begin{corollary}
Let $f:\C \rightarrow \C$ be quasiregular. Then $B(f)$ is a discrete set of points. In particular, if $f$ is quasiregular of polynomial type,
then $B(f)$ is a finite set of points.
\end{corollary}

\section{Canonical form}

It is well-known that every quadratic polynomial is linearly conjugate to $z^{2}+c$ for an appropriate $c \in \C$. We will find an analogous canonical form for compositions of quadratic polynomials and affine stretches.

Consider a quasiconformal mapping $h_{K,\theta}:\C \rightarrow \C$ which stretches by a factor $K >0$ in the direction $e^{i\theta}$. If $\theta =0$,
then $h_{K,0}(x+iy) = Kx+iy$. For general $\theta$, pre-compose $h_{K,0}$ by a rotation of $-\theta$ and post-compose by a rotation of $\theta$ to give the expression
\begin{equation*}
h_{K,\theta}(x+iy) = x(K\cos ^{2} \theta + \sin ^{2} \theta) + y(K-1)\sin \theta \cos \theta 
\end{equation*}
\begin{equation}
\label{eq3.4}
+ i \left [ x(K-1)\cos \theta \sin \theta + y (K \sin ^{2} \theta + \cos ^{2} \theta ) \right ]
\end{equation}
or
\begin{equation}
\label{eq3.4a}
h_{K,\theta}(z) = \left ( \frac{ K+1}{2} \right ) z + e^{2i \theta}\left ( \frac{ K-1}{2} \right )\zbar.
\end{equation}
Using the formula for complex dilatation (see \cite{FM}), we see that 
\begin{equation}
\label{eq3.3}
\mu _{h_{K,\theta}} = e^{2i\theta} \frac{K-1}{K+1},
\end{equation}
and so $\av \av \mu_{h_{K,\theta}} \av \av _{\infty} <1$ which means that $h_{K,\theta}$ is quasiconformal with constant complex dilatation. 
Solving the Beltrami equation $f_{\zbar} = \mu f_{z}$ (see \cite{FM}) with the complex dilatation of (\ref{eq3.3}) gives a quasiconformal map which is unique if we require the solution the fix three points.
Therefore given $K>0$ and $\theta \in[-\pi,\pi]$, the unique solution of the Beltrami equation with dilatation (\ref{eq3.3}) which fixes $0, \infty$ and $e^{i(\theta + \pi/2)}$ is the stretch given by (\ref{eq3.4}).

Given a stretch $h_{K,\theta}$, we will represent it as a point in $\C \setminus \{ 0 \}$, given by the point $Ke^{i\theta}$.
Note that a stretch of factor $1$ in any direction is just the identity.

\begin{proposition}
\label{canonicalform}
Let $f:\C \rightarrow \C $ be a composition of a quadratic polynomial and an affine stretch of the form (\ref{eq3.4}). Then $f$ is linearly conjugate to 
\begin{equation}
\label{eq3.8}
f_{K,\theta,C} := (h_{K,\theta})^{2}+C
\end{equation}
for some $Ke^{i\theta} \in \C \setminus \{ 0 \}$ and $C \in \C$.
Moreover, if we insist that $Ke^{i\theta} \in U$, where
\begin{equation}
\label{eq3.20}
U = \{ 0<\av z \av < 1, -\pi / 4 < \arg (z) \leq \pi / 4\} \cup \{ \av z \av > 1, -\pi / 4 < \arg (z) \leq \pi / 4\} \cup \{ 1 \},
\end{equation}
then such a representation is unique.
\end{proposition}

\begin{proof}
Let $g(z) = az^{2}+bz+c$ be a quadratic polynomial, where $a,b,c \in \C$ with $a \neq 0$, 
let $h=h_{M,\phi}$ for $M>0$ and $\phi \in [-\pi,\pi ]$, and write $f(z) = g(h(z))$. 
We need to know how $h$ behaves under pre-composition by translations and dilations.
Let $\alpha( z) = Az$ for $A \in \C \setminus \{ 0 \}$.
Then using (\ref{eq3.4a}),
\begin{equation*}
h(\alpha(z)) = h(Az) = \left ( \frac{ M+1}{2} \right ) Az + e^{2i \theta}\left ( \frac{ M-1}{2} \right )\overline{A}\zbar
\end{equation*}
\begin{equation}
\label{eq3.6}
= A \left ( \left ( \frac{ M+1}{2} \right ) z + e^{2i (\theta- arg(A))}\left ( \frac{ M-1}{2} \right )\zbar \right ) = Ah_{M,\theta - arg(A)}(z).
\end{equation}
Let $\beta(z) = z + B$ for $B \in \C$.
Again using (\ref{eq3.4a}), and noting that $h$ is $\R$-linear,
\begin{equation}
\label{eq3.7}
h(\beta(z)) = h(z) + h(B).
\end{equation}
Using (\ref{eq3.6}) with $A=1/a$,
\begin{equation*}
\alpha^{-1} \circ f \circ\alpha(z) = a \left ( a h_{M,\phi}(z/a)^{2} + b h_{M,\phi}(z/a) +c \right )
\end{equation*}
\begin{equation*}
= h_{M,\phi + arg(a)}(z)^{2} + b h_{M,\phi + arg(a)}(z) +ac
= \left ( h_{M,\phi + arg(a)}(z) + \frac{b}{2} \right )^{2} +ac - \frac{b^{2}}{4}.
\end{equation*}
Applying (\ref{eq3.7}) with $B = h^{-1}_{M,\phi + \arg(a)}(-b/2)$,
\begin{equation*}
\beta^{-1} \circ \alpha^{-1} \circ f \circ \alpha \circ \beta (z)
= h_{M,\phi + arg(a)}^{2}(z) + ac - \frac{b^{2}}{4} - h^{-1}_{M,\phi + arg(a)}(-b/2).
\end{equation*}
Therefore $f$ is linearly conjugate to (\ref{eq3.8}) with $K=M$, $\theta = \phi + \arg(a)$ and $C = ac - b^{2}/4 - h^{-1}_{K,\phi + arg(a)}(-b/2)$. 

For the uniqueness, 
we note that the choice of $K>0$ and $\theta \in [-\pi ,\pi ]$ for a given constant complex dilatation (\ref{eq3.3}) is not unique. There are the obvious symmetries
$(\theta \mapsto \theta + \pi)$ and $(K \mapsto 1/K, \theta \mapsto \theta + \pi /2)$. These correspond to the facts that a stretch in a direction $\theta$ is the same as stretching in the direction $-\theta$, and that stretching a factor $K$ in the direction $\theta$ is the same, up to post-composing by a conformal dilation, as stretching by a factor $1/K$ in the perpendicular direction. That is, $h_{K,\theta} = h_{K,-\theta}$
and $h_{K,\theta} = Kh_{1/K,\theta + \pi /2}$. There are no other such symmetries.

We see that $f_{K,\theta,C}$ is linearly conjugate to $f_{1/K,\theta +\pi/2, CK^{2}}$ via the conjugation $L(z) = z/K^{2}$. So if $Me^{i(\phi + \arg(a))} \notin U$, then we can apply such a conjugation (and take $-\phi - \arg(a)$ if necessary) so that $1/M e^{i(\phi +\arg (a) +\pi /2)} \in U$.

Recalling that all stretches with $K=1$ are identical, we see that if we require $Ke^{i\theta} \in U$, then the canonical form for $f_{K,\theta,C}$ is unique.
\end{proof}

For brevity, we will define the set
\begin{equation*}
QA = \{ f:\C \rightarrow \C : f=g \circ h, g(z)=z^{2}+c, h=h_{K,\theta}, c\in \C, Ke^{i\theta} \in U \}.
\end{equation*}
We will also mention the set
\begin{equation*}
PA = \{ f:\C \rightarrow \C : f=g \circ h, g \text{ is a polynomial of degree at least }2, h =h_{K,\theta} \},
\end{equation*}
and so $QA \subset PA$.

We observe that the dynamics of $f$ in Proposition \ref{canonicalform} correspond to those of $\widetilde{f}$, in particular, $z \in I(f)$ if and only if $z \in I(\widetilde{f})$.
Therefore, every composition of a quadratic polynomial and an affine stretch is linearly conjugate to some element of $QA$, and so we just need to study the dynamics of mappings in $QA$. 

\section{The boundary of the escaping set}

The escaping set is defined as
\begin{equation*}
I(f) = \{ z \in \C : f^{n}(z) \rightarrow \infty \}.
\end{equation*}
Recalling Theorem \ref{FN}, the requirement that the distortion is smaller than the degree is a necessary one as the following example shows.

\begin{ex}
\label{ex1}
Consider the winding map $f:(r,\theta) \mapsto (r, 2\theta)$ in polar coordinates.
This map decomposes as $f= g \circ h$ where $g(z)=z^{2}$ and $h(r,\theta ) = (r^{1/2},\theta)$.
The partial derivatives of $h$ are
\begin{equation*}
h_{r} = \frac{e^{i\theta }}{2r^{1/2}}, \:\:\:\:\: h_{\theta} = ir^{1/2}e^{i\theta}
\end{equation*}
and so the complex dilatation is
\begin{equation*}
\mu _{h} = e^{2i\theta} \frac{ h_{r} + \frac{i}{r}h_{\theta}}{h_{r} - \frac{i}{r}h_{\theta}} = \frac{-e^{2i\theta}}{3}.
\end{equation*}
We conclude that $\av \av \mu _{h} \av \av _{\infty} = 1/3$ and the distortion of $h$ is $2$. Therefore the distortion of $f$ is $2$, and clearly the degree of $f$ is $2$, but $I(f)$ is empty, since $\av f(z)\av =\av z\av$ for every $z \in \C$.
\end{ex}

Therefore, Theorem \ref{FN} only applies to those $f \in QA$ with $\av (K-1)/(K+1) \av <1/3$ by (\ref{eq3.3}). However, in our special situation, we can actually deduce the results of Theorem \ref{FN} by modifying the proof. 
For $f \in QA$, write $f=g \circ h$, where $g(z)=z^{2}+c$ and $h=h_{k,\theta}$ for $K e^{i\theta} \in U$ and $c \in \C$. We first estimate the growth of $h$.

\begin{lemma}
\label{hlemma}
Let $h$ be as above and set $L = \max \{ K, 1/K \}$.
Then for any $z \in \C$,
\begin{equation*}
\frac{\av z \av}{L} \leq \av h(z) \av \leq L \av z \av.
\end{equation*}
Further, $h$ is $L$-bi-Lipschitz, that is
\begin{equation*}
\frac{\av z - w\av}{L}   \leq \av h(z) - h(w) \av \leq L \av z - w \av
\end{equation*}
for all $z,w \in \C$.
\end{lemma}

\begin{proof}
This follows immediately from the definition of $h$, since
\begin{equation*}
L_{1}\av z \av \leq \av z \av \leq L_{2} \av z \av
\end{equation*}
where $L_{1} = \min \{K,1\}$ and $L_{2} = \max \{K,1\}$.
The fact that $h$ is bi-Lipschitz follows from the $\R$-linearity of $h$.
\end{proof}

We extend Theorem \ref{FN} in dimension $2$ as follows.

\begin{theorem}
\label{FNext}
Let $g$ be a polynomial of degree $d \geq 2$ and let $h$ be $L$-bi-Lipschitz. Let $f=g \circ h$, then $I(f)$ is a non-empty open set and $\partial I(f)$ is a perfect set.
\end{theorem}

\begin{proof}
A bi-Lipschitz mapping is quasiconformal, and so $f$ is quasiregular.
The first step is to show that $I(f)$ is non-empty. 
Since $h$ is $L$-bi-Lipschitz and $g(z) = z^{d}(1+o(1))$ for large $\av z \av$, we have
\begin{equation*}
\av f(z) \av = \av g(h(z)) \av  = \av h(z) \av ^{d} (1+o(h(z))) \geq \frac{\av z \av ^{d}}{L^{d}} (1+o(z)).
\end{equation*}
Therefore there exists $R>0$ such that 
\begin{equation*} 
\av f(z) \av >2 \av z \av,
\end{equation*} 
for $\av z \av >R$, and we can conclude that this neighbourhood of infinity is contained in $I(f)$.
The openness of $I(f)$ follows from the continuity of quasiregular mappings and the fact that $I(f)$ contains a neighbourhood of infinity. Clearly $I(f)$ is completely invariant, and therefore $\partial I(f)$ is completely invariant.

The fact the $I(f)$ is open means that $I(f)$ has no isolated points. To show that $I(f)$ is perfect, we therefore have to show that $I(f)^{c}$ has no isolated points. This is the part of the proof that requires modification when compared to Theorem \ref{FN}. Exactly as in the proof of that theorem 
(we omit the details here, see \cite{FN}), we assume for contradiction that if $z_{0}$ is an isolated point of $I(f)$, and see that then $I(f) = \C \setminus \{ z_{0} \}$.

Since $f(z_{0}) = z_{0}$, we must have $i(z_{0},f) = d$, where $i$ is the local index. Therefore
\begin{equation*}
f(z) = g(h(z)) = (h(z)-h(z_{0}))^{d} + z_{0}.
\end{equation*}
Using again the fact that $h$ is bi-Lipschitz,
\begin{equation*}
\av f(z) - z_{0} \av = \av h(z) - h(z_{0})\av ^{d} \leq L^{d}\av z - z_{0} \av ^{d} < \frac{\av z- z_{0} \av}{2}
\end{equation*}
for $\av z-z_{0} \av <2^{1/(1-d)}L^{d/(1-d)}$.
nd so there is a neighbourhood of $z_{0}$ which is not in $I(f)$, giving a contradiction.
\end{proof}

We remark that in the Stoilow decomposition of the function in Example \ref{ex1}, it is easy to see that the quasiconformal mapping is not bi-Lipschitz.

\begin{corollary}
\label{QAescset}
Let $f \in QA$. Then $I(f)$ is a non-empty open set, and $\partial I(f)$ is a perfect set.
\end{corollary}

\begin{proof}
This is an immediate consequence of Lemma \ref{hlemma} and Theorem \ref{FNext}.
\end{proof}

The following theorem is proved in \cite{FN}, and the proof goes through exactly as there and so is omitted.

\begin{theorem}
\label{list}
Let $f$ be as in Theorem \ref{FNext}. Then for any $k \geq 2$, $I(f^{k}) = I(f)$. The family of iterates $\{f^{k}:k \in \N \}$ is equicontinuous on $I(f)$ and not equicontinuous at any point of $\partial I(f)$. The set $\partial I(f)$ is infinite. The sets $I(f)$, $\partial I(f)$ and $\overline{I(f)} ^{c}$ are all completely invariant. The escaping set is a connected neighbourhood of infinity.
\end{theorem}

In particular, we have the conclusions of Theorem \ref{list} for $f \in QA$. One may be tempted to ask whether the proof that $\partial I(f)$ is perfect goes through as soon as $I(f)$ contains a neighbourhood of infinity. 
A modification of the winding map shows that this is not the case.

\begin{ex}
For $\lambda >1$ and $k \in \N$ with $k \geq 2$, define $f(re^{i \theta}) = \lambda r e^{i k\theta}$. This mapping decomposes as $f=g \circ h$ where $g(z)= z^{k}$ and
$h(z) = (\lambda r )^{1/k} e^{i \theta}$. We can calculate that the distortion of this mapping is $k$, that every point except $0$ escapes, and so $\partial I(f) = \{ 0 \}$. Clearly $\partial I(f)$ is not a perfect set. We note that $h$ is not bi-Lipschitz.
\end{ex}

\section{Connectedness of $\partial I(f)$}

Let $f$ be quasiregular of polynomial type and either the distortion of $f$ is smaller than the degree or $f \in QA$. Then we define the set of points whose orbits remain bounded
\begin{equation*}
N(f) = \{ z \in \C : \av f^{n} (z) \av <T , \text{ for some } T < \infty, \forall n \in \N \}.
\end{equation*}
Clearly $N(f) = I(f)^{c}$ and $N(f)$ is completely invariant by Theorem \ref{list}. This set is the direct analogue of the filled-in Julia set $K_{f}$ for polynomials, but here we are reserving the use of the symbol $K$ for distortion. 

Recall the branch set $B(f)$ is the set where $f$ is not locally injective. In this section, we will give proofs for $f \in QA$, but the proofs will work equally well for $f \in PA$ or when the degree of $f$ is larger than the distortion.
We first need a quasiregular version of the Riemann-Hurwitz formula.

\begin{theorem}[Riemann-Hurwitz formula]
Let $D_{1},D_{2}$ be domains in $\overline{\C}$ whose boundaries consist of a finite number of simple closed curves. Let $f(z)$ be a proper holomorphic map of $D_{1}$ onto $D_{2}$ with $L$ branch points including multiplicity. Then every $z \in D_{2}$ has the same number $d$ of pre-images including multiplicity and
\begin{equation}
\label{rhformula}
(2-d_{1}) = d (2-d_{2}) - L,
\end{equation}
where $d_{j}$ is the number of boundary components of $D_{j}$.
\end{theorem}

\begin{corollary}
\label{rhcor}
The Riemann-Hurwitz formula (\ref{rhformula}) holds when $f:D_{1} \rightarrow D_{2}$ is a quasiregular mapping of degree $d$ in $D_{1}$.
\end{corollary}

\begin{proof}
By Theorem \ref{Stoilow}, we can write $f=g \circ h$, where $g$ is holomorphic and $h$ is quasiconformal. Since $h$ is quasiconformal, it does not alter the number of boundary components of $D_{1}$, $h$ has no branch points in $D_{1}$ and every $z \in h(D_{1})$ has exactly one pre-image under $h$ in $D_{1}$. Therefore every $z \in D_{2}$ has $d$ pre-images under $g$ in $h(D_{1})$ including multiplicity, $g$ has $L$ branch points including multiplicity in $h(D_{1})$ and so we can apply (\ref{rhformula}) to $g:h(D_{1}) \rightarrow D_{2}$.
\end{proof}

We now move on to our first connectedness result.

\begin{figure}
  \caption{The set of non-escaping points for $f(z) = (h_{0.8,0}(z))^{2}-1.25$, a pinched basilica.}\label{0.8K_-1.25cfig}
\begin{center}
    \includegraphics{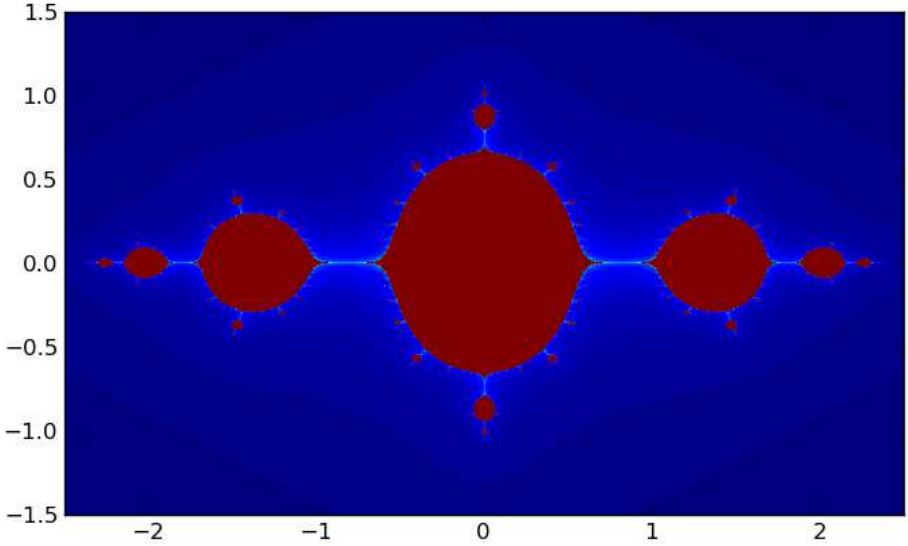}
\end{center}
\end{figure}

\begin{figure}
  \caption{The set of non-escaping points for $f(z) = (h_{0.7,\pi /12}(z))^{2}-2.297+0.295i$, a banking airplane.}\label{0.7_pi-12K-2.297+0.295icfig}
\begin{center}
    \includegraphics{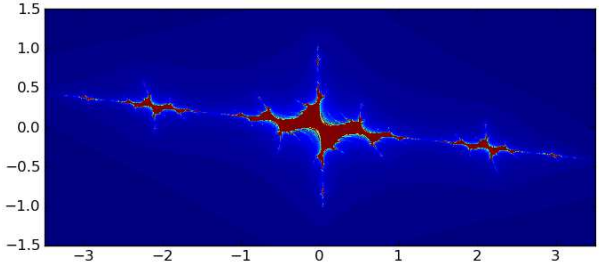}
\end{center}
\end{figure}

\begin{theorem}
\label{QAconn}
Let $f \in QA$. Then $N(f)$ is connected if and only if $I(f) \cap B(f) = \emptyset$.
\end{theorem}

\begin{proof}
Recall from the proof of Theorem \ref{QAescset} the existence of a neighbourhood of infinity $U=\{ \av z \av >R \}$ contained in $I(f)$.

First assume that $I(f) \cap B(f) = \emptyset$. Then $f^{j}$ gives a quasiregular $2^{j}$-sheeted covering of $U_{j} = f^{-j}(U)$ onto $U$ with no branching,
since $f^{j}$ has degree $2^{j}$. Since $U \cup \{ \infty \}$ is simply connected, so is every $U_{j} \cup \{ \infty \}$. Since $I(f) = \cup _{j=0}^{\infty} f^{-j}(U)$, then $I(f)$ must be simply connected and hence $N(f)$ is connected.

Conversely, suppose that $I(f) \cap B(f)$ is not empty. Let $M$ be the smallest value of $j$ such that $U_{j} \cap B(f) $ is not empty.
Applying Corollary \ref{rhcor} to the quasiregular map 
\begin{equation*}
f: (U_{M} \setminus \overline{U_{M-1}}) \rightarrow (U_{M-1} \setminus \overline{U_{M-2}})
\end{equation*}
with $d_{2} = 2$, $d=2$ and $L \geq 1$, we see that $d_{1} \geq 3$. This implies that $U_{M}$ is not simply connected and hence $N(f)$ is not connected.
\end{proof}

Since $0$ is the only branch point of $f \in QA$, this theorem is only interested in whether $0$ escapes or not. However, we have formulated it this way, because the proof is equally valid for $f \in PA$ or when the degree of $f$ is larger than the distortion, both cases where $f$ may have more than one branch point.
Examples of connected $N(f)$ are the Douady dragon of Figure \ref{0.8K_-0.21-0.78icfig}, the pinched basilica of Figure \ref{0.8K_-1.25cfig} and the banking airplane of Figure \ref{0.7_pi-12K-2.297+0.295icfig}.

Moving back to the case $f \in QA$, if $N(f)$ is not connected, then the following result tells us that $N(f)$ is actually infinitely connected.

\begin{theorem}
\label{infconn}
Let $f \in QA$. If $I(f)$ contains $B(f)$, then $N(f)$ is infinitely connected.
\end{theorem}

\begin{proof}
Since $B(f) \subset I(f)$, the forward orbit of $B(f)$,
\begin{equation*}
B^{+}(f) = \bigcup _{j=1}^{\infty} f^{j}(B(f))
\end{equation*}
accumulates only at infinity. Hence there exists a simple closed curve $C$ such that the interior $\Omega _{1}$ of $C$ contains $N(f)$ and the exterior $\Omega _{2}$ contains $B^{+}(f)$.
Since $\Omega _{2}$ is a neighbourhood of infinity and contained in $I(f)$,
\begin{equation*}
I(f)= \bigcup _{j=1}^{\infty} f^{-n} ( \Omega _{2})
\end{equation*}
and hence there exists $M$ such that $f^{M}(B(f) \cup \Omega_{2}) \subset \Omega _{2}$.
This implies that $f^{-M}$ has $2^{M}$ single valued locally injective quasiregular branches on $\Omega _{1}$, denoted by $h_{j} $ for $j=1,...,2^{M}$.

Then the sets $\{ U_{j} = h_{j}(\Omega _{1})\}$ are pairwise disjoint, and we call this collection of sets $X_{1}$. Each $U_{j}$ is compactly contained in $\Omega_{1}$ and contains $2^{M}$ images
$\{U_{j,k} = h_{j}(U_{k}) \}$ for $k=1,...,2^{M}$, and we denote the collection of such sets over all $j,k$ by $X_{2}$. 
One can inductively define the collection $X_{n}$ for every $n \in \N$, and by construction, 
\begin{equation*}
N(f) = \bigcap _{n=1}^{\infty} X_{n}.
\end{equation*}
Since $X_{n+1} \subset X_{n}$, we see that $N(f)$ is infinitely connected.
\end{proof}

For quadratic polynomials, if $I(f)$ contains $B(f)$, then $\partial I(f)$ is totally disconnected, i.e. every connected component is a point. This is not necessarily the case for $f \in QA$, as suggested by Figure \ref{0.8K_-1.1,0.003icfig} for $K=0.8$ and $c=-1.1+0.003i$, where each connected component appears to be a continuum. One can check that $0 \in I(f)$ by using the condition in Theorem \ref{thm6.3} below, and so $\partial I(f)$ really is disconnected, by Theorem \ref{infconn}, as claimed.

\begin{figure}
  \caption{The boundary of $I(f)$ for $K=0.8$, $\theta = 0$ and $c=-1.1+0.003i$}
  \label{0.8K_-1.1,0.003icfig}
  \begin{center}
    \includegraphics{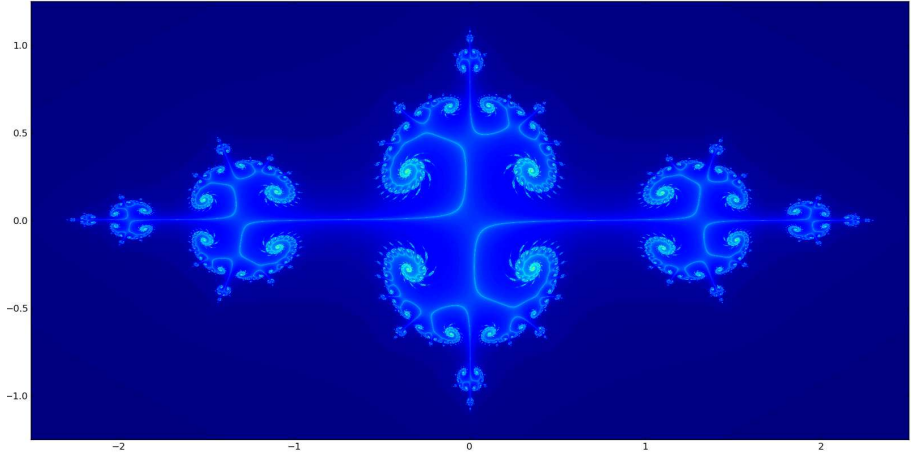}
  \end{center}
\end{figure}

\section{Parameter space}

For each pair $(K,\theta)$ such that $Ke^{i\theta} \in U$, we can consider the parameter space obtained by varying $c$ over functions in $QA$. We are led to the following definition.

\begin{definition}
Let $Ke^{i\theta} \in U$, and $f_{c} \in QA$ with decomposition $f_{c}=g_{c} \circ h$ where
$h= h_{K,\theta}$ and $g_{c}(z) = z^{2}+c$ for $c \in \C$. 
The $(K,\theta)$-Mandelbrot set $\mathcal{M}_{K,\theta}$ is defined be the set of those $c \in \C$ for which $0 \notin I(f_{c})$.
\end{definition}

\begin{theorem}
We have the following characterisation of $\mathcal{M}_{K,\theta}$:
\begin{equation*}
\mathcal{M}_{K,\theta} = \{ c \in \C : \partial I(f_{c}) \text{ is connected} \}.
\end{equation*}
\end{theorem}

\begin{proof}
By the canonical form for functions in $QA$, the only branch point of $f_{c}$ is $0$. By Theorem \ref{QAconn}, $0 \notin I(f_{c})$ if and only if $I(f_{c}) \cap B(f_{c})= \emptyset$ if and only if $N(f_{c})$ is connected if and only if $\partial I(f_{c})$ is connected. The last equivalence follows from Theorem \ref{list} and the fact that $I(f_{c})$ is a connected neighbourhood of infinity.
\end{proof}

\begin{figure}
  \caption{$(K,0)$-Mandelbrot sets for $K=0.7$ to $K=1.2$ in $0.1$ increments.}\label{kmandelbrotfig}
  \begin{center}
    \includegraphics{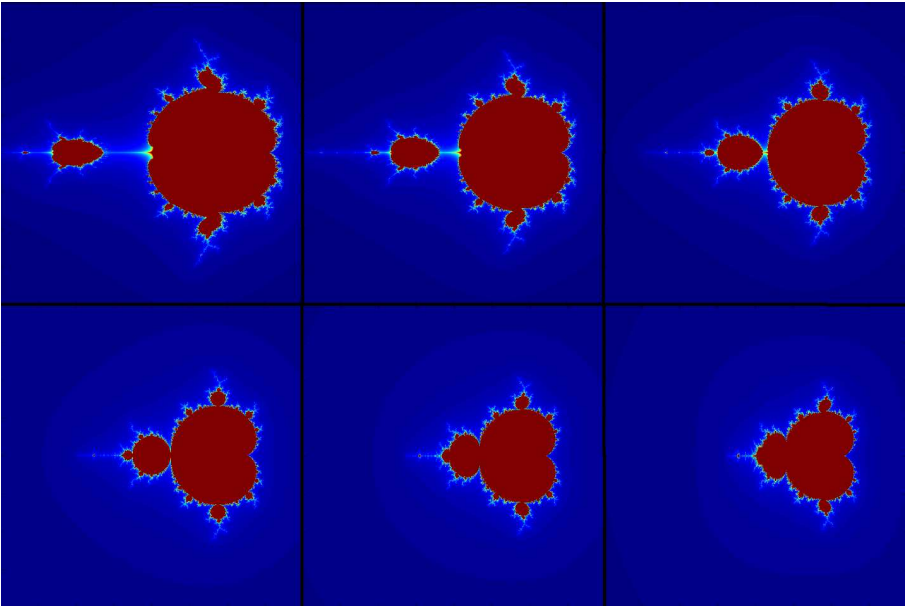}
  \end{center}
\end{figure}

\begin{figure}
\caption{$\mathcal{M}_{0.7,\pi/12}$}\label{M0.7_pi-12fig}
\begin{center}
\includegraphics{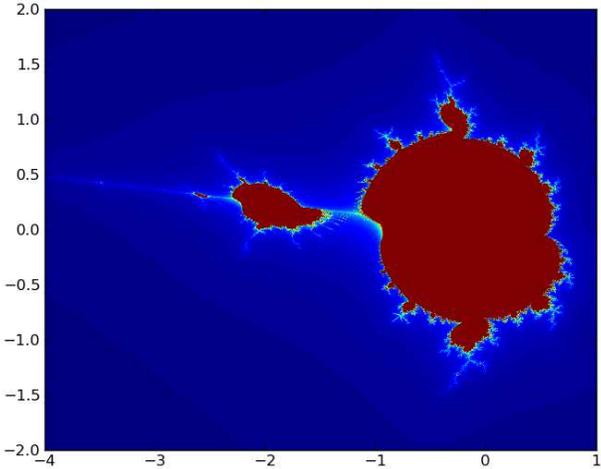}
\end{center}
\end{figure}

As is well-known, the Mandelbrot set $\mathcal{M}_{1,0}$ is contained in the disk $\{ \av c \av \leq 2\}$. We will show that the $K$-Mandelbrot set is also a bounded set. 

\begin{theorem}
\label{thm6.3}
Let $Ke^{i\theta} \in U$. Then
\begin{equation*}
\mathcal{M}_{K,\theta} \subset \{ c \in \C : \av c \av \leq 2 L_{1}^{-2} \},
\end{equation*}
where $L_{1} = \min \{K, 1 \}$.
Further, $\mathcal{M}_{K,\theta}$ is compact and can be characterised as the set of $c \in \C$ for which $f_{c}^{n}(0) 
\leq 2 L_{1}^{-2}$ for all $n \in \N$.
\end{theorem}

\begin{proof}
Fix $(K,\theta)$ such that $Ke^{i\theta} \in U$, and recall from Lemma \ref{hlemma} that $h = h_{K,\theta}$ satisfies $\av h(z) \av \geq L_{1} \av z \av$ for all $z \in \C$, where $L_{1} = \min \{K,1\}$.
Assume that $L_{1} ^{2}\av c \av  > 2$. Then $\av f(0) \av = \av c \av$ and
\begin{equation*}
\av f^{2}(0) \av = \av h(c)^{2} + c \av \geq \av c \av (L_{1} ^{2} \av c \av - 1).
\end{equation*}
For $n \geq 2$, assume that 
\begin{equation*}
\av f^{n}(0) \av \geq \av c \av (L_{1}^{2} \av c \av - 1)^{2^{n-2}}. 
\end{equation*}
Then
\begin{equation*}
\av f^{n+1}(0) \av \geq \av f ( f^{n}(0)) \av \geq \av h( f^{n}(0))^{2} + c \av
\end{equation*}
\begin{equation*}
\geq L_{1} ^{2}( \av c \av (L_{1} ^{2} \av c \av - 1)^{2^{n-2}})^{2} - \av c \av \geq \av c \av ( L_{1} ^{2} \av c \av - 1)^{2^{n-1}}
\end{equation*}
when $L_{1} ^{2} \av c \av > 2$. Therefore, by induction, $\av f^{n}(0) \av \rightarrow \infty$ and $c$ is not in the $(K,\theta)$-Mandelbrot set.

For the second part of the theorem, assume that 
\begin{equation*}
\av f_{c}^{m}(0) \av > 2L_{1}^{-2} + \epsilon > 2L_{1}^{-2} 
\end{equation*}
for some $m \in \N$ and $\epsilon >0$.
First note if $\av c \av = \av f_{c}(0) \av > 2L_{1}^{-2}$ then $c \notin \mathcal{M}_{K,\theta}$. 
So if $\av c \av \leq 2L_{1}^{-2}$, then
\begin{equation*}
\av f_{c}^{m+1}(0) \av = \av (f^{m}_{c}(0))^{2} +c \av \geq  \left ( 2L_{1}^{-2} + \epsilon \right ) ^{2} - 2L_{1}^{-2}
\end{equation*}
\begin{equation}
\label{eq6.1}
\geq 2L_{1}^{-2} \left ( 2L_{1}^{-2} - 1 \right ) +4L_{1}^{-2} \epsilon.
\end{equation}
If $K \geq 1$, then $L_{1} = 1$ and (\ref{eq6.1}) implies that $\av f_{c}^{m+1}(0) \av  \geq 2+4\epsilon$. Therefore by induction,
\begin{equation*}
\av f_{c}^{m+k}(0) \av \geq 2 + 4^{k}\epsilon
\end{equation*}
and we can conclude that $0 \in I(f_{c})$ and so $c \notin \mathcal{M}_{K,\theta}$.
On the other hand, if $K <1$, then $L_{1}^{-2} >1$ and (\ref{eq6.1}) implies that
$\av f_{c}^{m+1}(0) \av \geq 2L_{1}^{-2} + ( 2L_{1}^{-2})^{2}\epsilon$ and so by induction,
\begin{equation*}
\av f_{c}^{m+k}(0) \av \geq 2L_{1}^{-2} + (2L_{1}^{-1})^{2k} \epsilon
\end{equation*}
and again $c \notin \mathcal{M}_{K,\theta}$.
This argument shows that the complement of $\mathcal{M}_{K,\theta}$ is open and so $\mathcal{M}_{K,\theta}$ itself is a compact set.
\end{proof}

See Figure \ref{kmandelbrotfig} for various Mandelbrot sets with $\theta = 0$. It may appear that a bulb is detached from the main cardioid for $K<1$, but the following theorem shows that, among other things, these two components of $\mathcal{M}_{K}$ are attached by a segment contained in $\R$. See also Figure \ref{M0.7_pi-12fig} for this effect. 

\begin{theorem}
If $\theta = 0$, then 
\begin{equation*}
\mathcal{M}_{K,0} \cap \R = \left [-\frac{2}{K^{2}},\frac{1}{4K^{2}} \right ].
\end{equation*} 
If $\theta \neq 0$, then there exists an angle $\phi_{0} \in [0,2\pi]$ and a real number $\eta$ such that the line segment
\begin{equation*}
te^{i\phi_{0}} \subset \mathcal{M}_{K,\theta},
\end{equation*}
for 
\begin{equation*}
t \in \left[ -\frac{2}{\eta},\frac{1}{4\eta} \right ].
\end{equation*}
\end{theorem}

\begin{proof}

If $K \in  \R^{+}$, $\theta=0$ and $c \in \R$, then $f_{c}^{n}(0) \in \R$ for all $n \in \N$. It is well-known that $\mathcal{M}_{1} \cap \R = [-2, 1/4]$, i.e. $0$ does not escape under iteration of $x^{2}+c$ only when $c \in [-2,1/4]$. Since $f_{c}(x) = K^{2}x^{2}+c$ for $x \in \R$, then this observation implies that 
\begin{equation*}
\mathcal{M}_{K,0} \cap \R = \left [-\frac{2}{K^{2}},\frac{1}{4K^{2}} \right ].
\end{equation*}

For the second part, assume that $\theta \neq 0$, $K>0$ and set
$h=h_{K,\theta}$ and $f=h^{2}$. Clearly $h$ maps rays emanating from $0$ to other such rays and has obvious fixed rays. We first show that $f$, which also maps rays to rays, has at least one fixed ray too.

One can calculate that $h$ is given, in polar coordinates, by
\begin{equation}
\label{eq6.4}
h(re^{i\varphi}) = 
r\left ( 1+  (K^{2}-1)\cos ^{2}(\varphi - \theta ) \right )^{1/2} \exp \left [ i \left (\theta + \tan ^{-1} \left ( \frac{\tan (\varphi - \theta )}{K} \right ) \right ) \right ].
\end{equation}
For $\varphi \in [0,2\pi]$, write $L_{\varphi}$ for the ray
$\{ te^{i\varphi}: t\geq 0 \}$. Then it is easy to see that $f$ maps the ray $L_{\phi}$ onto $L_{T(\phi)}$, where
\begin{equation}
\label{eq6.2}
T(\phi ) = 2\theta +2 \tan ^{-1} \left ( \frac{\tan (\phi - \theta )}{K} \right ).
\end{equation}
To show that $f$ has a fixed ray, we therefore need to find a solution $\phi _{0}$ to $T(\phi) = \phi$. Rearranging (\ref{eq6.2}), this means we need to solve
\begin{equation}
\label{eq6.3}
\tan \left ( \frac{\phi}{2} - \theta \right) = \frac{\tan (\phi - \theta )}{K}.
\end{equation}
Let $t=\tan [(\phi - \theta)/2]$. Then (\ref{eq6.3}) and the addition formula for the tangent function yield
\begin{equation*}
\frac{t - \tan (\theta /2)}{1-t \tan (\theta /2)} = \frac{2t}{K(1-t^{2})}.
\end{equation*}
Rearranging this equation, we need to find a zero of the cubic polynomial 
\begin{equation*}
P(t):= Kt^{3} - (2+K)\tan (\theta /2 ) t^{2} +(2-K)t + K\tan (\theta / 2).
\end{equation*}
Every cubic with real coefficients has a real zero, and we claim that $P(t)$ has a zero $t_{0}$ in $[-1,1]$. To see this, note that $P(1) = 2(1-\tan (\theta/2))$ and $P(-1) = 2(-1-\tan(\theta/2))$. Since
$\theta \in [-\pi /4, \pi /4]$, we have $\tan (\theta/2) \in [-\sqrt{2}+1,\sqrt{2}-1]$ and the claim follows.
Therefore, with
\begin{equation*}
\phi _{0} = \theta + 2 \tan ^{-1} t_{0},
\end{equation*}
$f$ fixes $L_{\phi_{0}}$.
Using (\ref{eq6.4}), we see that $f$ acts on the fixed ray $L_{\phi_{0}}$ by
\begin{equation*}
f(re^{i\phi_{0}}) = 
r^{2}\left ( 1+  (K^{2}-1)\cos ^{2}(\phi_{0} - \theta ) \right ) e^{i\phi_{0}}.
\end{equation*}
The final claim of the theorem follows with 
\begin{equation*}
\eta = \left ( 1+  (K^{2}-1)\cos ^{2}(\phi_{0} - \theta ) \right )
\end{equation*}
by using the same argument from the first part of the theorem.
We remark that if $\theta = 0$, then $\phi_{0}=0$ and $\eta = K^{2}$, agreeing with the first part of the theorem.
\end{proof}

We remark that the parameter space for mappings in $QA$ is actually two complex dimensional, since we have 
a different mapping for each pair $(Ke^{i\theta},c)\in U \times \C$. Denote by $\Omega \subset U \times \C$
the set of $(Ke^{i\theta},c)$ such that $f_{K,\theta,c}^{n}(0)$ is bounded.
Each $(K,\theta)$-Mandelbrot set
is a one dimensional slice of this larger parameter space $\Omega$. Similarly, one can consider one dimensional slices of $\Omega$ where $c$ is fixed, and $Ke^{i\theta}$ varies.
As an example, Figure \ref{c_-1.5fig} is the slice of $\Omega$ through $c=-1.5$. Note the expected rotational symmetry. Further, the slice $c=0$ is the whole of $U$. A natural question to ask is whether the set $\Omega$ is connected in $\C^{2}$, just as the Mandelbrot set is connected in the parameter space of quadratic polynomials?

\begin{figure}
\caption{$c=-1.5$ parameter slice}\label{c_-1.5fig}
\begin{center}
\includegraphics{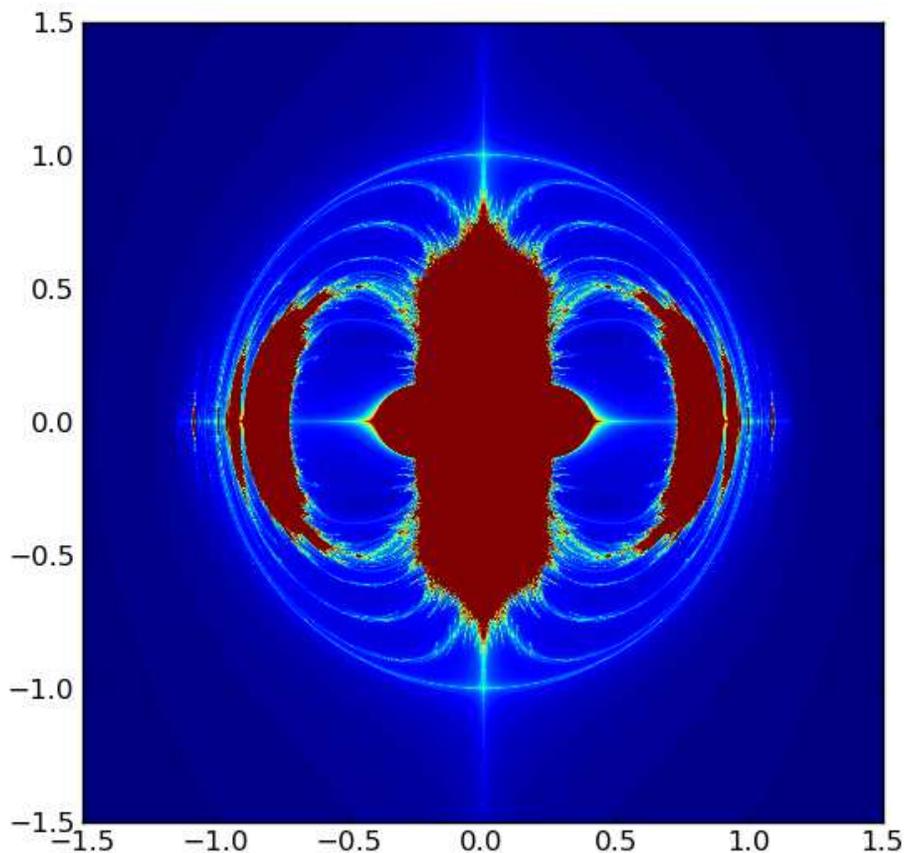}
\end{center}
\end{figure}

\small

Institute of Mathematics, University of Warwick, Coventry, CV4 7AL, UK. 

{\it Email address:} {\tt alastair.fletcher@warwick.ac.uk}

\medskip

Equipe Audition, D\'{e}partement d'Etudes Cognitives,
Ecole Normale Sup\'{e}rieure,
29 Rue d'Ulm,

75230, Paris, Cedex 05.

{\it Email address:} {\tt dan.goodman@ens.fr}

\end{document}